\documentclass[12pt]{amsart}
\usepackage{tikz,array, verbatim}
\usepackage{amsfonts, amsmath, latexsym, epsfig, caption}
\usepackage{amssymb, color}

\usepackage{epsf}
\usepackage{array}
\usepackage{ragged2e}
\usepackage{hyperref}
\usepackage[margin=1in]{geometry}
\usepackage{longtable}

\title{ Local Groups in Delone Sets }

\author[N. Dolbilin]{Nikolay Dolbilin}
\address{Nikolay Dolbilin, \newline  Steklov Mathematical Institute, \newline Gubkina str.8, Moscow, Russia 119991 \newline dolbilin@mi-ras.ru}

\begin{document}

\newcommand{\R}{\ensuremath{\mathbb{R}}}

\newtheorem{theorem}{Theorem}[section]
\newtheorem{proposition}[theorem]{Proposition}
\newtheorem{corollary}[theorem]{Corollary}
\newtheorem{lemma}[theorem]{Lemma}
\newtheorem{problem}[theorem]{Problem}
\newtheorem{conjecture}{Conjecture}
\newtheorem{question}{Question}
\newtheorem{claim}{Claim}

\newtheorem{definition}[theorem]{Definition}

\theoremstyle{remark}
\newtheorem*{remark}{Remark}

\begin{abstract} In the paper, we prove  that in  an arbitrary Delone set $X$ in $3D$ space, the subset $X_6$ of all points from $X$ at which  local groups  has axes of the order not greater than 6 is  also a Delone set.  Here, under the \textit{local group at point} $x\in X$ is meant the symmetry group $S_x(2R)$ of the cluster $C_x(2R)$  of $x$ with radius $2R$, where $R$ (according to  Delone's theory of the 'empty sphere') is the radius of the largest 'empty' ball, that is, the largest ball free of points of $X$.
 
The main result (Theorem~\ref{thm:dn6})   seems to be  the  first  rigorously proved statement on   \textit{absolutely generic} Delone sets which implies substantial  statements for Delone sets with strong crystallographic  restrictions. For instance,   an important observation   of Shtogrin on the boundedness of local groups in Delone sets with  equivalent $2R$-clusters (Theorem~\ref{thm:stogrin}.) immediately follows from Theorem~\ref{thm:dn6}.    

In the paper, the  'crystalline kernel  conjecture' (Conjecture 1) and its two weaker versions (Conjectures 2 and 3) are suggested.    
According to Conjecture 1,  in a quite arbitrary  Delone set, points with locally crystallographic axes (of order 2,3,4, or 6) only \textit{inevitably} constitute essential part of the set. These conjectures significantly generalize the famous statement  of Crystallography  on the  impossibility of (global) 5-fold symmetry in a 3D lattice.

 \end{abstract}

\maketitle


\section{Introduction and basic definitions}
\label{intro}

This paper  grew out  of the local theory for regular systems, i.e for Delone sets with  very strong requirements. On the other hand,  here we consider an arbitrary  Delone sets in $\mathbb R^3$ without any additional assumptions. For example,  for  a Delone set,   we do not suppose a  typical condition of the local theory such as the  sameness of clusters of  certain radius as we did it in numerous papers (see e.g.,  \cite{local}, \cite{Dol2018}, \cite{DGLS}).
Another feature of  the paper is as follows,  for a Delone set, we consider  local groups operated over  clusters of  radius, namely  $2R$ (for definitions see below). 

In \cite{mackay1986},  A.L.~Mackay says: "In an infinite  crystal there  may be extra elements of symmetry which operate over a limited range. These may be seen by non-space-group extinctions in diffraction pattern. \ldots The local operations need not to be 'crystallographic'."  Nevertheless,  in this paper,  we show     that the last statement for   regular systems (i.e. for sets with transitive groups) should be significantly revised.    If we fix for a Delone sets  in 3D space, of type $(r,R)$   the range of action of local  groups as $2R$ (for details see below),   then, it turns out,  one can obtain interesting results on properties of such groups.
To accurately  formulate the results and open  hypotheses we will need 
 several  definitions and notations. 
 
 Euclidean distance  between points $x$ and $x'$ in euclidean space $\mathbb R^d$ is denoted by  $|x,x'|$. Let  $d(z, X)$ denote distance from $z\in \mathbb R^d$ to set $X\subset \mathbb R^d$, i.e. $d(z, X):=\inf_{x\in X}  |z, x|$].  

\begin{definition}
\label{def:deloneset}[Delone set]
Given  positive real numbers $r$ and $R$, a point subset $X$ of $\mathbb{R}^d$ is called a \emph{Delone set of type $(r,R)$} if the following two conditions hold:
\\
(1) an open $d$-ball $B^o_z(r)$ of radius $r$ centered at any point $z$ of space  contains at most one point of $X$;
\newline
(2) a  closed $d$-ball $B_z(R)$ of radius $R$ centered at an arbitrary point $z$ of space contains at least one point of $X$.
\end{definition}

It is clear that  a Delone set  $X$ of type $(r,R)$ is obviously a Delone set of type $(r',R')$ if $r'\leq r$ and  $R'\geq R$. Therefore,  we can adopt the convention in designating  the parameters $(r,R)$ to a Delone set $X$ as follows. For a given Delone set $X$,  we choose as $r$ the largest possible value satisfying  Condition (1) of Definition~\ref{def:deloneset},  and choose as $R$ the smallest  value satisfying  Condition (2). 

We will need also the following interpretations of the parameters $r$ and $R$:
 \[\inf_{x,x'\in X}|xx'|=2r, \qquad  \sup_{z\in \mathbb R^d}d(z, X)=R. \eqno(1)\]
 
 Thus, the value of $r$ equals  the  half of the smallest (infimum) inter-point distance in $X$. The value of $R$ is  a  distance from the most remote from $X$ point of space $\mathbb R^d$ to the set  $X$.
 
   In the local theory of Delone sets, the key concepts are that of a cluster.

\begin{definition}[$\rho$-cluster]\label{def:cluster}
	
  Let  $x$ be a point of a Delone set $X$ of type $(r,R)$,  $\rho\geq 0$, and let  $B_z(\rho)$ be a ball with radius $\rho$ centered at point $z\in \mathbb R^d$.   We call a point set  
\[C_{\mathbf x}(\rho):=X\cap B_{\mathbf x}(\rho)\] 
the {\em cluster of radius $\rho$} at point $x$  or simply the {\em $\rho$-cluster  at $x$}.
\end{definition}

\begin{definition}[equivalent clusters] \label{def:clustequiv}
	Two clusters $C_x(\rho)$ and $C_{x'}(\rho)$ of the same radius $\rho$ at points $x$ and $x'$ are said to be \textit{equivalent} if there is an isometry $g\in Iso(\mathbb R^3)$ such that 
	\[ g(x)=x' \quad {\rm and} \quad  g(C_x(\rho))=C_{x'}(\rho).\eqno(2)\]    
\end{definition}

\begin{definition}[cluster group]\label{def:clustergroup}\label{def:clustergroup}
Given a point $x\in X$ and its $\rho$-cluster $C_x(\rho)$, a group $S_x(\rho)$ of all isometries $s\in Iso(\mathbb R^d)$ which leave the  $x$ fixed and the cluster $C_x(\rho)$ invariant  is called the \textit{cluster group}:
\[S_x(\rho):=\{s\in  Iso (\mathbb R^d)\,|\, s(x)=x, \, s(C_x(\rho ))=C_x( \rho)\}.\] 
\end{definition} 

Groups of equivalent clusters  $C_x(\rho)$ and  $C_{x'}(\rho)$ are conjugate in the full group $Iso(d)$ of isometries: 
$S_x(\rho)=g^{-1}\, S_{x'}(\rho)\,  g$, 
where   $g$  is determined by conditions of Definition~\ref{def:clustequiv} . It is clear that as the radius $\rho$ increases, the cluster $C_x(\rho )$ expands but  the cluster group $S_x(\rho)$ never increases and sometimes can    contract only.   It is clear that if $0\leq \rho< 2r$ group $C_x(\rho)= O_x(3)$? i.e. the full point group of all isometries that leave point $x$ fixed. On the other and, it is well-know that the  $S_x(2R)$

   Since the main result grew up ideologically from  the local theory of regular systems, here, we briefly recall basic concepts   of this theory.
   Modern in form, the following definitions of regular system and crystal are equivalent to those that go back to E.S.~Fedorov. 
   
   \begin{definition}[regular system, crystal]  A Delone set $X$ is called a {\it regular system}  if it is an orbit of some point $x$ with respect to a certain space group $G\subset Iso(d)$,    i.e.
   \[X=G\cdot x=\{g(x)\,|\,g\in G\}; \]  
   a Delone set $X$ is a {\it crystal}  if $X$ is a union of several orbits: $X=\cup_{i=1}^m G\cdot x_i$.
   \end{definition}

We emphasize that the notion of a regular system is an essential case of the  crystal, i.e. a multi-regular system, and generalizes the lattice concept. In fact, a lattice is a a particular case of a regular system when $G$ is a group of translations  generated by $d$ linearly independent translations. Moreover, due to a celebrated theorem by Schoenflies and Bieberbach, any regular system is the union of congruent and mutually parallel lattices.          

The local theory of regular systems  began with the Local criterion in \cite{local}.

\begin{theorem}[Local Criterion, \cite{local}] A Delone set $X$ is a regular system if and only if there is some $\rho_0>0$ such that the following two conditions hold: 
\newline 1) all $\rho_0+2R$-clusters are mutually equivalent;
\newline 2) $S_x(\rho_0)= S_x(\rho_0+2R)$ for $x\in X$. 
	
\end{theorem}

In \cite{DolSht1989}, \cite{DLS1998}, this criterion has been  generalized for crystals, i.e multi-regular systems.

From now on, we restrict ourselves only to the 3D case.  One of central problems of the local theory of regular systems is to search for an  upper (and lower) bound for the regularity radius,   i.e.  a minimum value   $\hat{\rho}_3>0 $ such  that  equivalence of ${\hat\rho}_3$-clusters in a Delone set $X$ implies the regularity of the set $X\subset \mathbb R^3$.  In \cite{Dol2018},\cite{DGLS} is given  a proof of the upper bound $\hat\rho_3\leq 10R$. The long proof starts with selection of a special finite   o of finite subgroups of $O(3)$. Groups of this list   have a chance to occur  in Delone sets with equivalent $2R$-clusters as local groups $S_x(2R)$. The list of  selected   groups  is provided by    Theorem~\ref{thm:stogrin} found by Shtogrin in the late 1970's but published only in 2010 (\cite{Sto2010}). 

 \begin{theorem}[\cite{Sto2010}]\label{thm:stogrin}
    If  in a Delone set $X\in \mathbb R^3$ all $2R$-clusters are mutually equivalent,  then  the  order of any rotational axis of $S_x(2R)$ does not exceed 6.    
   	   \end{theorem}
  
  Quite recently, \cite{Dol2019},  it was realized for the first time that   an  important statement  about  groups in Delone sets \textbf{with significant requirements on equivalent clusters} may follow from a certain  statement true for  \textbf{ pretty general} Delone sets.
  
   Namely, in \cite{Dol2019},  Theorem~\ref{thm:dn6-19}   is proved. 
    Given an arbitrary Delone set $X\subset \mathbb R^3$ and $x\in X$, let  the maximal order of  rotational axis in group $S_x(2R)$  be denoted by $n_x$. 
  
  \begin{theorem}\label{thm:dn6-19} In a Delone  set $X\subset \mathbb R^3$ there is at least one point $x$ with $n_x\leq 6$.  	
  \end{theorem}

It is obvious that Theorem~\ref{thm:dn6-19} immediately implies Theorem~\ref{thm:stogrin}  In fact, the subset $X_6$ of all points in a Delone set $X$  with $n_x\leq 6$ is always very rich. Due to Theorem~\ref{thm:dn6} the subset $X_5$  is a Delone set itself.

\section{The  main result and conjectures}

\begin{theorem}[Main result]\label{thm:dn6}
  Given a Delone set   $X\subset \mathbb R^3$  of type $(r,R)$,  let $ X_6\subseteq X$ be the  subset of all points $x\in X$such that  the maximal order $n_x$  of a  rotation axes in  $S_x(2R)$ does not exceed  $6$.  Then $X_6$ is  a Delone set of a certain type   $(r',R')$, where $r\leq r'\leq R'=kR$ for some $k$ independent on $X$.  	
\end{theorem} 
At the moment, we do not care on the value $kR$ of the upper bound  for the parameter $R'$ of the $X_6$. For us,  so far it is more important to establish that the subset $X_6$ is
 always  a Delone subset.

From now on, we will focus on clusters $C_x(2R)$ of radius $2R$ and their groups $S_x(2R)$. As well-known, for a Delone set $X$,  the $2R$-clusters all are full-dimensional (i.e. the dimension of their convex hulls is $d$). Hence, the cluster  groups $S_x(2R)$ are  necessarily finite.    At the same time, we emphasize that the value of $2R$ is the smallest value of radius that guarantees the finiteness of the  cluster group of   radius $2R$ for any Delone set with parameter $R$. In other words, for an arbitrary $\varepsilon>0$ one can present   a Delone set $X$ with parameter $R$ such that for some $x\in X $  in $S_x(2R-\varepsilon)$ is infinite.

By virtue of the above, we will single out group $S_x(2R)$ and call it a \textit{local group at   $x$}.

\medskip

  Theorem~\ref{thm:dn6}   immediately  implies  Theorem~\ref{thm:stogrin} which concerns a Delone set $X$ with mutually equivalent $2R$-clusters.   Really, for such a Delone set $X$,  the local groups at  all points are pairwise conjugate and existence of   points $x$ with $n_x\leq 6$ implies the same inequality for all others.

  Now, among points of $X_6\subseteq X$ we select  points $x$  with the condition $n_x\neq 5$, i.e., all points of $X$ whose local groups contain axes of only 'crystallographic' orders 2,~3,~4, or 6. We call the  subset  of all such points in $X$ a \textit{crystalline kernel} of $X$ and denote by $K$.

  \begin{conjecture}  [Crystalline kernel conjecture]\label{conj:cryst}
   The  crystalline kernel $K$ of a Delone set $X$ is a Delone subset with some parameter $R'\leq kR$, where $k$ is some constant which  does not depend on $X$.   
  	
  \end{conjecture}
  
Let   $Y$ denote  the subset of all points $x\in X$ at which  local groups $S_x(2R)$ do not contain  the pentagonal axis. It is clear that   $K\subseteq Y$ and if $K$ is a Delone set then $Y$ is a Delone set too. Therefore   
Conjecture~\ref{conj:cryst}, if proven,  immediately implies the following two 
Conjectures 2 and 3.

\begin{conjecture}[5-gonal symmetry conjecture]\label{conj:pentag}
Given a Delone set $X\subset \mathbb R^3$, the subset $Y$ of points $x$, whose  groups $S_x(2R)$ are free of 5-fold axes,  is also a Delone set.
	\end{conjecture}  

 In its turn, Conjecture~2 enforces the following  statement that  seems to be much easier proved.

\begin{conjecture}[Weak 5-gonal symmetry conjecture]\label{conj:weak-pentag}
	Given a Delone set $X\subset \mathbb R^3$ with mutually  equivalent $2R$-clusters,  the local group $S_x(2R)$ contains no 5-fold axis.  
\end{conjecture}

It is obvious that these hypotheses relate to a celebrated  crystallographic theorem on the impossibility  of the global 5-fold symmetry in a three-dimensional lattice. Conjectures~1--3  significantly reinforce a classical statement  on famous crystallographic restrictions.   

It is well-known that for a  3-dimensional lattice even in  the group $S_x(r_1)$,   there are no the 5-fold symmetry, where $r_1=2r\leq 2R$ is the minimum inter-point distance in the lattice. In contrast to  lattices, in regular systems in 3$D$ space, the pentagonal symmetry can \textit{locally} manifest itself on clusters of a certain radius less than $2R$. 
So, for  instance,   there are regular systems such that even  group  $S_x(r_3)$   contains the 5-fold axis, but the local group $S_x(2R)$ does not. Here, in the systems,  $r_3~ (r_1<r_2<r_3<2R$) is the 3rd inter-point distance in $X$.
    But, it is still unknown 
whether there are regular systems with 5-fold symmetrical $2R$-clusters.

Since   the regularity radius for dimension 3 is not less than $6R$ (see \cite{LowerBound}),   among Delone sets $X$ with mutually equivalent $2R$-clusters, there are  non-regular and even   non-crystallographic sets. Thus,  even the weakest   Conjecture~3 concerns also a wide class of those non-regular sets.


\section{Proof of Theorem~\ref{thm:dn6}}

\begin{proof} Generally speaking, in the local group $S_x(2R)\subset O(3)$, $ x\in X$, there are several axes of maximal order $n_x$. Bearing in mind the well-known list of all finite subgroups of $O(3)$, we see that more than one axes of the maximal order $n_x$ in $S_x(2R)$ cannot happen  provided $n_x>5$. Let $\ell_x$ be one of those axes. Since  rk$(C_ x(2R))=3$, i.e. the convex hull of the $2R$-cluster is 3-dimensional, in  $C_x (2R)$, there are necessarily  points off the $\ell_x$.

Since $ X_6$ is a   subset of $X$, the minimal  inter-point distance $2\tilde r$ in $X_6$ (in fact, the infimum of such distances) is not less than  $ 2r=\inf_{x,x'\in X}|x,x'|$. In order to prove that $X_6$ is a Delone set with a certain parameter  $\tilde R$,  we will prove that the distance from a given  point  $z\in \mathbb R^3$   to the nearest point of $X_6$ does not exceed $\tilde R$: $\min_{x\in X'}|z,x|\leq \tilde R$ (due to interpretations (1) for the parameters $r$ and $R$ in Section 1. 
We will be looking for the point  $x\in  X_6$ nearest  to $z\in \mathbb R^3$ by  walking  along a special finite point sequence in $X_6$.

\begin{definition}
	
A sequence of points $[x_1,x_2, \ldots x_m, \ldots] \in X$  (finite or infinite, no matter`) is termed an \textit{off-axial chain} if the following condition holds for any $i=1,2,\ldots$:
  \newline the point $x_{i+1}\in X$ is the nearest point to  $x_i$ among all points of $X$ which are off the axis $\ell_{x_i}$, where  the axis $\ell_{x_i}$ means an axis of the local group $S_{x_i}(2R)$ of the maximal order $n_{x_i}$.   In case the subgroup of all (orientation-preserving) rotations of  $S_{x_i}(2R)$ is trivial (that is, in the local group at $x_i$ no  axes through $x_1$),  any nearest  to $x_i$ point of $X$	 can be chosen as $x_{i+1}$.  
  	\end{definition}
  	
  	Note  that for any point $x\in X$,  there are   off-axial sequences $[x_1(=x),x_2,x_3,\ldots ]$. 
  	
  	\begin{lemma}\label{lem:chain}
  		Given a Delone set $X$ and an off-axial sequence $[x_1,x_2,\ldots , x_m,\ldots ]\subset X$,  assume that  $x_i\notin X_6$, $\forall\, i\in \overline{1,m}$. Then for $i\in \overline{1,m}$ the following holds: 
  		\[|x_i, x_{i+1}|<0.87^{i-1}\cdot 2R \quad  {\rm and} \quad |x_1,x_m|<7.7\cdot 2R=15.4\,R, {\rm \, for\,  all\, }\, m. \eqno(3)  \] 
  	\end{lemma}
  	   \begin{proof}(of Lemma~\ref{lem:chain}).  
	Let $[x_1,x_2,,\ldots , x_m,\ldots ]$ be an off-axial chain and assume that it belongs $X\setminus X_6$.  
	Recall that, by construction, in this chain,  the length $r^*_i$ of each link $x_i,x_{i+1}$ is less than $2R$. Hence $x_{i+1}\in C_{x_i}(2R)$. Therefore, the rotation $g_{x_i}\in S_{x_i}(2R)$  can be applied to point $x_{i+1}$ too.

 By assumption, $x_1\notin X_6$, that is,   $n_{x_1}\geq 7$. Let $x_2$ be the nearest to $x_1$  point $x_2$ which is off the axis $l_{x_1}$, also let   $g_{x_1}$ be   
    a rotation  around axis $\ell_{x_1}$ by angle $2\pi/n_{x_1}$. Since 
    $r_1^*\leq 2R$ the   cluster  $C_{x_1}(2R)$ necessarily contains vertices of a regular $n_{x_1}$-gon $P_1$   which is generated by  the rotation $g_{x_1}\in S_{x_1}(2R)$ applied to the point $x_2$. The polygon $P_1$ is  located in a plane orthogonal to the $\ell_{x_1}$. The center of $P_1$ is  on $l_{x_1}$ (Figure~\ref{fig:distance}). 

    Denote the side-length of $P_1$ by $a_1$. Since the circumradius of $P_1$ does not exceed $r_{x_1}^*$ and  $n_{x_1}\geq 7$ we have  for $a_1$ and  $r_{x_2}^*$  the following estimate
\[r_{x_2}^*\leq a_1\leq 2r_{x_1}^*~\sin \frac{\pi }{\,n_{ x_1}}\leq 2r_{x_1}^*~\sin \frac{\pi}{7}<0.87~r_{x_1}^*< 0.87\cdot 2R. \eqno (4)\]

   
   

\smallbreak

\begin{center}
\begin{figure}[!ht]	
	\includegraphics[width=4.5cm]{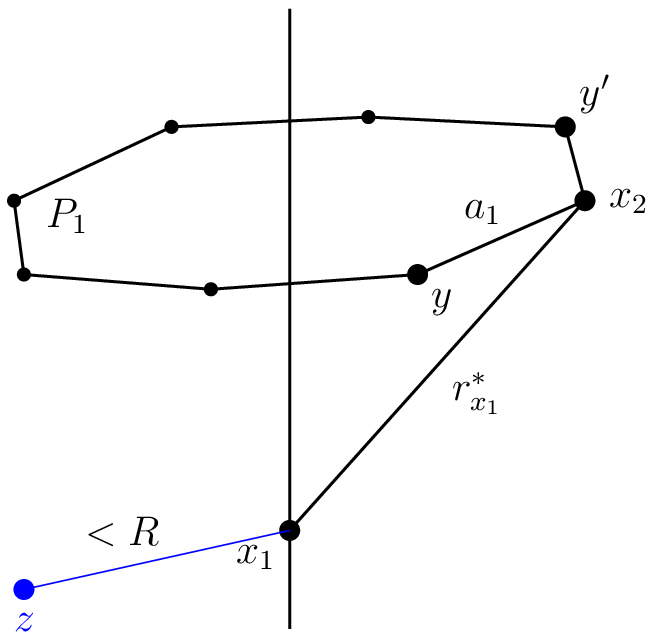}
	\end{figure}
\end{center}

\captionof{figure}{Polygon $P_1$ and the beginning of an  off-axial chain   $[x_1, x_2,\ldots ]$ \label{fig:distance}}


  \medskip
   Assuming now  that   $x_2\notin X_6$,  i.e. $n_{x_2}\geq 7$, we  will construct  the next point  $x_3$ in  the  off-axial chain $[x_1,x_2,\ldots]$ and obtain upper estimates  for $r_{x_3}$ and $a_2$ (see inequalities (5) below).
    
The  rotation $g_{x_2}$ about  the axis $\ell_{x_2}$ is assumed to belong to the local group  $S_{x_2}(2R)$.  Since $x_3\in C_{x_2}(2R)$, the $g_{x_2}$   can be applied to the point $x_3$. Hence  
the  cluster  $C_{ x_2}(2R)$ necessarily contains vertices of a regular $n_{x_2}$-gon $P_2$ generated by  rotation $g_{x_2}$ applied to the point $x_3$.
 Denote the side-length of $P_2$ by $a_2$ and  note that $a_2\leq r_2^*\leq 2R$.
 
Point  $x_3$, like  a vertex of the  regular $n_{x_2}$-gon $ P_2$,   has two adjacent vertices in $P_2$ at distance $a_2$ from vertex $x_3$.   Vertex $x_3$ and two   adjacent vertices of $P_2$ form a non-collinear triple. Therefore, no matter how the axis $\ell_{x_2}$ passes through the point $x_2$,  anyway,  at least one  of two neighboring points  is off   the axis $\ell_{x_2}$. It follows that the distance $r_{x_3}^*$ from $x_3$ to the nearest point $x_4\in X\setminus \ell_{x_3}$ does not exceed  $a_2$.  Since  we bear in mind that in the set $X\setminus \ell_{x_2}$, there may be points nearer to the $x_3$ than distance $a_1$ we have
$r_{x_2}^*\leq a_2\ $.   

 Since $n_{x_2}\geq 7$, by the same argument as above,  we have 
\[r_{x_3}^*<a_2\leq 2r_{x_2}^*~\sin \frac{\pi }{n_{ x_2}}\leq 2r_{x_2}^*~\sin \frac{\pi}{7}<0.87~r_{x_1}^*< 0.87^2\cdot 2R. \eqno (5) \]

This reasoning can be repeated over and over again. Under condition  $n_{x_i}\geq 7$, $\forall i\in \overline{ 1,m}$,  we get the off-axial chain $[x_1,x_2,\ldots ]$  for which the sequence of inter-point distances  $r_{x_i}^*=   |x_ix_{i+1}|$ is dominated by   a geometric progression

\[ r_{x_{i+1}}^*<0.87~r_{x_i}^*<(0.87)^i~r_{x_1}^*<(0.87)^i~2R. \eqno(6)\]
  
  Thus, we obtain the required inequalities (3).  Lemma~\ref{lem:chain} is proved. 
  \end{proof}	 

\bigskip

Now we are going to complete the proof of Theorem~\ref{thm:dn6}. Given a Delone set $X$, from Lemma~\ref{lem:chain} it follows that any off-axial chain provided $n_{x_i}\geq 7$ is finite. Moreover its length $m$ can be bounded from above: $m<M:= \log (\frac{R}{r}) /\log \frac{1}{0.87}.   $   
This implies that in the chain $[x_1,x_2,\ldots ]$  there are points $x_m$. with $n_{x_m}\leq 6$, i.e. $x_m\in X_6$.  By Lemma~\ref{lem:chain} the segment-length $|x_1, x, x_m|<15.4$. 

Now  we set up    upper boundedness of the distance from an arbitrary point $z$  of space to the nearest point of the subset $X_6$. Let  the nearest to $z$ point of $X$ be $x_1$ (see Figure~\ref{fig:distance}) and $x_m\in X_6$ (by Lemma~\ref{lem:chain}). Then $|z,x_1|\leq R$ and we get   
\[\min_{x\in X_6} |z,x|\leq   |z, x_m|\leq |z,x_1|+|x_1,x_m|=16.4\,R. \eqno(7) \] 

In other words, we proved that the subset $X_6$ is a Delone set with some parameter $\tilde R\leq 15.4\,R.$  
Theorem~\ref{thm:dn6} is proved. 
\end{proof}

\section{Concluding remarks}
We emphasize that the upper bound   for the parameter $\tilde R$,  established here, is far from optimal one. We believe  that we will soon be able  to present a  sharper bound \cite{DolSto2020} . The purpose of this  paper was to present the`  result, in our opinion, of a new type.  The result   suggests a few conjectures which should be interesting both in itself and  in  context of the theory of quasicrystals. So, for instance,  in Penrose patterns, in structures of real Shechtman quasicrystals, the centers of $2R$-clusters with   local 5-fold symmetry constitute a rich Delone subset. At the same time, in these known quasicrystalline structures, there are also Delone subsets of points with local crystallographic axes (including identical)    `
 However, according to  Conjecture~1 , not only in in these structures  but  in any other possible Delone sets, points with local crystallographic axes \textbf{inevitably} constitute an essential part of the structure.

\section{Acknowledgments}
This work  was supported by the Russian Science Foundation under grant 20-11-20141 and performed in Steklov Mathematical Institute of Russian Academy of Sciences.

\end{document}